\documentclass[10pt]{amsart}
\usepackage{mystyle} 
\usepackage{color}

\newcommand{\bx}{\mathbf{x}}
\newcommand{\bp}{\mathbf{p}}
\newcommand{\bq}{\mathbf{q}}

\newcommand{\bfN}{\mathbf{N}}

\newcommand{\balpha}{\boldsymbol{\alpha}}
\newcommand{\bbeta}{\boldsymbol{\beta}}

\newcommand{\dd}{\; \mathrm{d}}

\newcommand{\bbn}{\mathbb{N}}
\newcommand{\bbz}{\mathbb{Z}}
\newcommand{\bbr}{\mathbb{R}}

\newcommand{\sep}{\;:\;}

\author{Erez Nesharim, Rene R\"uhr, Ronggang Shi}
\title{Metric Diophantine approximation with congruence conditions}
\date{}

\thanks{*E. Nesharim was supported by EPSRC Programme Grant: EP/J018260/1. R. R\"{u}hr was supported by the S.N.F., project number 168823. R. Shi is supported by NSFC 11871158. This project has received funding from the European Research Council (ERC) under the European Union’s Horizon 2020 research and innovation program (grant agreement No. 754475).}

\begin{document}
\maketitle

\begin{abstract}
We prove a version of the Khinchine--Groshev theorem for Diophantine approximation of matrices subject to a congruence condition. The proof relies on an extension of the Dani correspondence to the quotient by a congruence subgroup. This correspondence together with a multiple ergodic theorem are used to study rational approximations in several congruence classes simultaneously. The result in this part holds in the generality of weighted approximation but is restricted to simple approximation functions.
\end{abstract}

\section{Introduction}

For positive integers $m$ and $n$,  let $\mathrm{M}_{m,n}(\R)$ denote the space of real $m\times n$-matrices. Let $\|\cdot\|$ denote the maximum norm on $\R^k$ for any positive integer $k$. An \emph{approximation function} is a function $\psi:\bbn\to(0,\infty)$.
One of the  main theorems in Metric Number Theory is the following:

\begin{theorem}[Khinchin--Groshev Theorem]\label{thm:K} 
Let $\psi$ be a 
non-increasing approximation function. If
\begin{equation}\label{eq:divergence}
\sum_{n=1}^\infty\psi(n)=\infty
\end{equation}
then for almost every $ \theta \in \mathrm{M}_{m,n}(\R)$ there are infinitely many $(\bp,\bq)\in\bZ^m\times \bZ^n$ satisfying
\begin{equation}\label{eq:psiApproximation}
\| \theta  \mathbf{q} + \mathbf{p}\|^m\leq \psi\left(\|\mathbf{q}\|^n\right).
\end{equation}
\end{theorem}
Here and later in this note vectors are denoted by bold letters and if $\bx\in\R^k$ its coordinates are denoted by $x_1,\ldots,x_k$. Note that if $(m,n)\neq(1,1)$ then Theorem~\ref{thm:K} holds without the monotonicity assumption, see \cite{beresnevichVelani2010}. Our first theorem is a variation of Theorem~\ref{thm:K} in which solutions subject to coordinate-wise congruence conditions. Set $d=m+n$.

\begin{theorem}\label{thm:KAlongProgressions}
Let $\psi$ be a non-increasing approximation function, and fix $\bfN\in\N^d$ and $\bv\in\Z^d$. If $\psi$ satisfies~\eqref{eq:divergence} then for almost every $ \theta \in \mathrm{M}_{m,n}(\R)$ there are infinitely many $(\mathbf{p},\mathbf{q})\in\bZ^m\times \bZ^n$ satisfying \eqref{eq:psiApproximation} such that $p_j = v_j \Mod N_j$ for every $1\leq j\leq m$ and $q_j = v_{m+j} \Mod N_{m+j}$ for every $1\leq j\leq n$.
\end{theorem}

Diophantine approximation with congruence conditions was first studied by Hartman--Sz\"usz
\cite{hartmanszusz}. They proved a special case of  Theorem~\ref{thm:KAlongProgressions} where $m=n=1$ and the congruence condition on $\bp$ is trivial (i.e.~$N_2=1$) using the Duffin--Schaeffer Theorem.
Nontrivial congruence restrictions on $\bp$ were first considered by Harman \cite{harman} who proved Theorem~\ref{thm:KAlongProgressions} in the case $m=n=1$ using Harmonic Analysis. See also the more recent work of Adiceam \cite{adiceam2015progressions}.

Since a countable intersection of full measure sets has full measure,  it follows from Theorem~\ref{thm:KAlongProgressions} that almost every matrix has infinitely many approximations in every congruence class. Our next theorem reveals a new connection between approximations in several congruence classes for approximation functions of the form $\psi(n)=\frac{\eps}{n^\delta}$ where $\eps>0$ and $\delta\in(0,1]$.

\begin{theorem}\label{thm:A}
Assume $c>0$, $\delta<1$, $\ell\in\bbn$, and let $\bfN_i\in\N^d$ and $\bv_i\in\Z^d$ for every $1\leq i\leq \ell$. Then for almost every $ \theta \in \mathrm{M}_{m,n}(\R)$ there exist infinitely many $Q\in \bbn$ such that there exists a collection $\left\{(\mathbf{p}_i,\mathbf{q}_i)\right\}_{1\leq i\leq \ell}\subseteq\bZ^m\times \bZ^n$ that satisfy the following:
\begin{itemize}
\item $\|\bq_i\| \leq Q$ for every $1\leq i\leq \ell$.
\item $\|\theta\mathbf{q}_1+\mathbf{p}_1\|^m\leq c Q^{-n}$.
\item $\|\theta\mathbf{q}_i+\mathbf{p}_i\|^m\leq c Q^{- \delta n}$ for every $2\leq i\leq \ell$.
\item $p_{i,j} =v_{i,j}  \Mod N_{i,j}$ and $q_{i,r} =v_{i,m+r}  \Mod N_{i,m+r}$ for all $1\leq i \leq \ell$, $1\leq j \leq m$ and  $1\leq r \leq n$.
\end{itemize}
\end{theorem}

The proofs of Theorem~\ref{thm:KAlongProgressions} and Theorem~\ref{thm:A} use a translation of Diophantine conditions on a matrix into dynamical properties of  a certain trajectory in a homogeneous space. This method, initiated by Dani \cite{dani}, was explored further by Kleinbock--Margulis \cite{kleinbockmargulis2} who gave a dynamical proof of Theorem~\ref{thm:K}. We develop this method further to study Diophantine approximations with congruence conditions.

A crucial ingredient for the proof of Theorem~\ref{thm:A} is the pointwise multiple ergodic theorem proved in \cite{shi2015expanding}.
The generality dealt with in \cite{shi2015expanding} allows us
to obtain a similar result for Diophantine approximation with weights. 
A \emph{$k$-dimensional weight} is a probability vector in $\bR^k$ with positive entries. If $\balpha\in\bbr^k$ is a weight denote the \emph{quasi-norm associated with $\balpha$} by $\|\cdot\|_{\balpha}$ and define it by
\[
\|\mathbf{x}\|_{\balpha}=\max \left\{|x_j|^\frac{1}{\alpha_j} \sep 1\leq j \leq k
\right
\}
\]
for every $\bx\in\bbr^k$. Diophantine approximation with weights from the homogeneous dynamics viewpoint were considered in \cite{kleinbock1998flows}, and studied further e.g.\ in
\cite{kleinbock2017pointwise} and \cite{kleinbock2008dirichlet}.
Theorem~\ref{thm:A} can now be refined as follows:
\begin{theorem}\label{thm:B}
Assume $\eps>0$, $\ell\in\bbn$, and $\bfN_i\in\N^d$, $\bv_i\in\Z^d$ for every $1\leq i \leq \ell$. Let $\balpha_i\in \bR^m$ and $\bbeta_i\in \bR^n$ be weights for every $1\leq i \leq \ell$. Assume that $\kappa_1,\ldots,\kappa_\ell$ are positive  real numbers that satisfy
\begin{equation}\label{eq:unnecessaryAssumption}
\kappa_i\left(\balpha_i,\bbeta_i\right) - \kappa_{i-1}\left(\balpha_{i-1},\bbeta_{i-1}\right) > 0 \text{ for all } 2\leq i\leq \ell.
\end{equation}
Then for almost every $\theta\in \mathrm{M}_{m,n}(\R)$ there exist arbitrarily large $Q\in \R$ such that for every $1\leq i\leq \ell$ there exists $\left(\bp_i,\bq_i\right)\in\bZ^m\times \bZ^n$ that satisfy:
\begin{itemize}
\item $\|\theta\mathbf{q}_i+\mathbf{p}_i\|_{\balpha_i}\leq \eps Q^{- \kappa_i}$ and $\|\bq_i\|_{\bbeta_i} \leq \eps Q^{\kappa_i}$.
\item $p_{i,j} =v_{i,j}  \Mod N_{i,j}$ for every $1\leq j \leq m$ and $q_{i,j} =v_{i,m+j}  \Mod N_{i,m+j}$ for every $1\leq j \leq n$.
\end{itemize}
\end{theorem}

\begin{proof}[Proof of Theorem~\ref{thm:A}]
Apply Theorem~\ref{thm:B} with
\[
\left(\balpha_i,\bbeta_i\right)=\left(1/m,\ldots,1/m,1/n,\ldots,1/n\right)
\]
for every $1\leq i\leq \ell$, $\eps=\frac{c}{2^n}$, $\kappa_1=1$ and any distinct $\kappa_2,\ldots,\kappa_\ell\in(\delta,1)$, and take $\lceil Q\rceil\in\bbn$ for $Q>1$ as in the conclusion of Theorem~\ref{thm:B}. 
\end{proof}

\begin{remark}\leavevmode

\item(i) It follows from a general zero-one law \cite{beresnevichVelaniZeroOneLaws} that if Theorem~\ref{thm:B} holds for $\eps=1$ then it holds for every $\eps>0$. However, during the course of its proof in \S~\ref{sec:proof} it will be evident that it is more natural to state it in this form.

\item(ii) The conclusion of the theorem does not in general hold when all weights are equal and \eqref{eq:unnecessaryAssumption} is removed: Indeed, for $m=n=1$, any $\theta\in\bbr$ and any positive $Q\in\bbr$, all the solutions $(p,q)\in\bbz^2$ to $|\theta q+p|<Q^{-1}$ and $|q|<Q$ are integer dilations of a single solution. With more effort it is possible that our method could be applicable to certain shrinking targets, replacing $\eps$ with a function of $Q$ that decays slower then an exponential function.

\item(iii) In the case that all the weights are distinct it remains an interesting problem to decide whether or not \eqref{eq:unnecessaryAssumption} is a necessary condition.

\end{remark}

\section{Dani correspondence with congruence condition}

In this section Dani correspondence is refined to encode a congruence condition.
Let $N$ be a positive integer. Let $G=\SLR[d]$ be the group of real $d\times d$-matrices of determinant one
and $\Gamma=\SLZ[d]$ be the subgroup of all integer matrices in $G$.
It is well-known that the homogeneous space $X=G/\Gamma $ can be naturally identified with the set of unimodular lattices in $\mathbb R^{d}$ through the map  $g\Gamma\in X\mapsto g\bZ^{d}$.

Let $\balpha\in\bbr^m$ and $\bbeta\in\bbr^n$ be weights, and let $\|\cdot \|_{\balpha, \bbeta}$ be the weighted quasi-norm on $\mathbb R^d$
defined by
$$\|\mathbf v \|_{\balpha, \bbeta} =\max \left(\|\mathbf x \|_{\balpha}, \|\mathbf y  \|_{\bbeta}\right),$$
where $\mathbf v =(\mathbf x, \mathbf y) $ and $ \mathbf x\in \mathbb R^m, \mathbf y\in \mathbb R^n$.
For any $t\in\bbr$ let
\[
a_{\balpha,\bbeta}(t)=\diag{e^{t\alpha_1},\dots, e^{t\alpha_m}, e^{-t\beta_1},\dots, e^{-t\beta_n }}.
\]
For each $\theta \in\Mat_{m,n}(\R)$ 
associate a matrix
\begin{equation}\label{eq:uA}
u(\theta) =\begin{bmatrix}
\Id_m& \theta  \\
0 & \Id_n
\end{bmatrix}
\end{equation}
in $G$. 

Recall that a vector $\mathbf v\in \mathbb Z^{d}$ is said to be \emph{primitive} if it is nonzero and the greatest common divisor among its coefficients is one. 
Let $\widehat {\mathbb Z^{d}}$ denote the set of primitive vectors of $\mathbb Z^{d}$.
For every  $\mathbf v\in \widehat {\mathbb Z^{d}}$
  let
$$\mathcal P_N(\mathbf v)=\left\{\mathbf w\in \widehat {\mathbb Z^{d}}\;:\;\mathbf w=\mathbf v\Mod N  \right\}.$$

It is not hard to see that  $\widehat {\mathbb Z^{d}}$ is a disjoint union of $\mathcal P_N(\mathbf v)$ for $  \mathbf v$ taken over a set of representatives of $\widehat {\mathbb Z^d} \Mod N$.

Let
\[
\Gamma_N=\left\{\gamma\in \Gamma\; :\; \gamma \mathbf e _1 =\mathbf e _1 \Mod N\right\},
\]
where $\mathbf e_1=(1, 0,\ldots, 0)\in\bbz^d$, and set $X_N=G/\Gamma_N$. Note that $\Gamma_N$ is a finite index subgroup of $\Gamma$. 
The group $\Gamma$ acts transitively on $\widehat {\mathbb Z^d}$
and the coset decomposition of $\Gamma/ \Gamma_N$ corresponds to
the congruence decomposition of $\widehat {\mathbb Z^d}$, namely
\[
\gamma \Gamma_N \mathbf e_1=\mathcal P_N \left(\gamma \mathbf e_1\right) \text{ for all } \gamma \in \Gamma .  
\]

Let $K=\SOR[d]$ be the subgroup of all orthogonal matrices in $G$, let $U$ be the upper triangular unipotent subgroup, and for any $\varepsilon>0$ let
\[
A_\varepsilon=\left\{\diag {s_1, \ldots, s_d}\in G\;:\;  s_1 < \varepsilon,\; s_j>0 \text{ for all } 1\le j\le d  \right\}.
\]
The Iwasawa decomposition of the group $G$ implies that $KA_\varepsilon U$ is a nonempty open subset of $G$. Therefore, 
 the set defined by
\[
X_{N}^{\varepsilon}= \{g\Gamma_N: g\in KA_\varepsilon U\Gamma_N   \}
\]
is a nonempty open subset of $X_N$.

\begin{lemma}\label{lem;criterion}
Suppose $0<\varepsilon<1$ and $\gamma\in \Gamma$.
If $g\gamma\Gamma_N \in X_N^{\varepsilon}$, then there exists $\mathbf v\in
\mathcal P_N(\gamma \mathbf  e_1) $ such that $\|g \mathbf v \|_{\balpha, \bbeta}< \varepsilon $.
\end{lemma}

\begin{proof}
Since $g\gamma\Gamma_N \in X_N^{\varepsilon}$, there exists $\tilde{\gamma}\in \Gamma_N$ such that
$g\gamma\tilde{\gamma}\in KA_\varepsilon U$.
Let $\mathbf v=\gamma\tilde{\gamma}\mathbf e_1 \in \mathcal P_N(\gamma \mathbf  e_1)$.
Then
\[
\|g\mathbf v\|_{\balpha,\bbeta}= \|g\gamma\tilde{\gamma}\mathbf e_1\|_{\balpha, \bbeta} = \|k a u \mathbf e_1\|_{\balpha, \bbeta},
\]
where $k, a, u$ belong to $K, A_\varepsilon, U$, respectively.
Suppose $a=\mathrm{diag}(s_1, \ldots, s_d)$, then
  $kau \mathbf e_1=s_1k \mathbf e_1$ so its Euclidean norm
is less than $\varepsilon$ according to the definition of $A_\varepsilon$. Hence, $\|g\mathbf v\| < \eps$. Since $\balpha$ and $\bbeta$ are weights, then $1/\alpha_j$ and $1/\beta_j$ are greater than one for all $1\leq j\leq m$ and $1\leq j\leq n$,
respectively. Therefore, $\|g\mathbf v\|_{\balpha,\bbeta} < \varepsilon$.
\end{proof}

\begin{corollary}\label{cor;main}
Let $\gamma \in \Gamma, \theta\in \mathrm{M}_{m, n}(\mathbb R), t>0$
and $0<\varepsilon<1$. If $a_{\balpha,\bbeta}(t)u(\theta) \gamma\Gamma_N\in X_N^{\varepsilon}$, then there exists $(\mathbf p, \mathbf q)\in \mathcal P_N(\gamma\mathbf e_1)$ where
$\mathbf p \in \mathbb Z^m$ and $  \mathbf q \in \mathbb Z^n$ such that
\begin{equation}\label{eq;diop}
\| \theta  \mathbf{q}+\mathbf{p}\|_{\balpha} < \varepsilon e^{-t} \quad
\mbox{and}	\quad
\|\mathbf q \|_{\bbeta} < \varepsilon e^t.
\end{equation}
\end{corollary}

\begin{proof}
According to Lemma~\ref{lem;criterion}, there exists $\mathbf v = (\mathbf p, \mathbf q)\in \mathcal P_N(\gamma\mathbf e_1)$ such that 	
\begin{equation}\label{eq;temp}
\| a_{\balpha,\bbeta}(t)u( \theta )\gamma\mathbf v \|_ {\balpha, \bbeta}<\varepsilon.
\end{equation}
A direct calculation shows that (\ref{eq;temp}) is equivalent to
(\ref{eq;diop}).
\end{proof}

\section{A proof of Theorem~\ref{thm:B}}\label{sec:proof}

The proof is an application of the following special case of Theorem A.1 from \cite{shi2015expanding}.

\begin{theorem}\label{thm;main}
Let $\ell\in\bbn$, let $\balpha_i\in \bR^m$ and $\bbeta_i\in \bR^n$ be weights for every $1\leq i \leq \ell$ and assume that $\kappa_1,\ldots,\kappa_\ell$ are real numbers that satisfy \eqref{eq:unnecessaryAssumption}. Let $N\in\bbn$ and let $\mu$ denote the probability Haar measure on $X_N$. Then for any $\varphi_1,\ldots,\varphi_\ell\in C_c^\infty\left(X_N\right)$ and $g_1,\dots,g_\ell\in G$, one has for almost every $\theta\in \mathrm{M}_{m, n}(\mathbb R)$
\begin{equation}\label{eq;main}
\lim_{T\to \infty}	\frac{1}{T}\int_0^T \prod_{i=1}^\ell \varphi_i\left(a_{\balpha_i,\bbeta_i}\left(\kappa_i t \right)u(\theta)g_i\Gamma_N\right)\dd t =\prod_{i=1}^\ell\mu\left(\varphi_i\right).
\end{equation}
\end{theorem}
\begin{proof}[Proof of Theorem~\ref{thm:B}]

Without loss of generality assume that there exists $N$ such that $N=N_{i,j}$ for all $1\leq i \leq \ell$, $1\leq j\leq m$. Otherwise, replace each $N_{i,j}$ with $N=\operatorname{lcm}\left(N_{1,1},\ldots,N_{\ell,d}\right)$.

Similarly, without loss of generality assume that $\gcd \left(\bv_i,N\right)=1$ for every $1 \leq i \leq \ell$. Otherwise, replace $\bv_i$ with $\frac{\bv_i}{\gcd\left(\bv_i,N\right)}$ for every $1 \leq i \leq \ell$ and replace $\eps$ with $\frac{\eps}{N^r}$, where 
\[
r=\max_{i,j} \left\{1/\alpha_{i,j}, 1/\beta_{i,j} \right\}.
\]

For every $1\leq i\leq \ell$ choose $\gamma_i\in \Gamma$ such that
\[
\gamma_i \mathbf e_1=\bv_i \Mod N.
\]
Recall that the set $X_N^{\varepsilon}$ is a nonempty open subset of $X_N$. Therefore, there exists a compactly supported smooth function $\varphi: X_N\to [0, \infty)$ such that $\mu(\varphi)>0$ and the  support of $\varphi$ is contained in $ X_N^{ \varepsilon}$.

Now suppose that $\theta \in \mathrm{M}_{m, n}(\mathbb R) $ satisfies \eqref{eq;main}. Then there exists a strictly increasing sequence of integers $\left\{ t_k \right\}_{k=1}^\infty$ such that
$$
\varphi\left(a_{\balpha_i,\bbeta_i}\left(\kappa_it_k\right)u(\theta) \gamma_i\Gamma_N\right)>0
$$
for all $1\le  i\le \ell$ and any $k\in\bbn$.
Since the support of $\varphi_i$ is contained in $X_N^{\varepsilon}$, this implies that
\[
a_{\balpha_i,\bbeta_i}\left(\kappa_it_k\right)u(\theta)\gamma_i\Gamma_N\in X_N^{\varepsilon}.
\]
for all $1\le  i\le \ell$ and any $k\in\bbn$. Hence, by Corollary~\ref{cor;main}, for every  $k$ and every $1\le i \le \ell$ there exists
$(\mathbf p_i, \mathbf q_i)\in \mathcal P_N\left(\gamma_i \mathbf e_1 \right)$ such that
$$
	\| \theta  \mathbf{q}_i+\mathbf{p}_i\|_{\balpha} < \varepsilon e^{-\kappa_i t_k} \quad
\mbox{and}	\quad
\|\mathbf q_i \|_{\bbeta} < \varepsilon e^{\kappa_i t_k}.
$$
Moreover, since $(\mathbf p_i, \mathbf q_i)\in\mathcal P_N(\gamma_i \mathbf e_1 )$, it satisfies
$(\mathbf p_i, \mathbf q_i)= \mathbf v_i \Mod N$. Therefore, setting $Q=e^{t_k}$ 
proves that there are arbitrarily large $Q\in\bbr$ satisfying the conclusion of Theorem~\ref{thm:B}.

\end{proof}

\section{A proof of Theorem~\ref{thm:KAlongProgressions}}\label{sec:proof2}

In this section it is shown how Theorem~\ref{thm:KAlongProgressions} falls into the framework of Kleinbock--Margulis \cite{kleinbockmargulis2}. As their theorems are rather general, a simplified and self-contained exposition is possible.

Given $\bfN\in\N^d$ and $\bv\in\Z^d$ as in the statment of Theorem~\ref{thm:KAlongProgressions}, it is enough to consider a single modular level $N$ (i.e.\ all $N_i=N$) as was done in the beginning of the proof of Theorem~\ref{thm:B} and that $\bv$ is primitive.

Now specify $\gamma\in\Gamma$ such that $\gamma \mathbf e_1=\mathbf v$.
For an approximation function $\psi$ and any $\gamma\in\Gamma$ define $x=g\Gamma_N\in X_N$ to be \textit{$\gamma$-congruence-$\psi$-approximable} if there exists $(\mathbf v,\mathbf w)\in g\gamma \Gamma_N\mathbf e_1$ with arbitrarily large $\|\mathbf w\|$ such that $\|\mathbf v\|^m\leq \psi(\|\mathbf w\|^n)$. Hence our task is to show that for any $\gamma\in\Gamma$, almost every $\theta\in M_{m,n}(\R)$, $u(\theta)\Gamma_N$ is $\gamma$-congruence-$\psi$-approximable.

Denote $x^\gamma=g\gamma \Gamma_N$ and define the function $\Delta$ on $X_N$ by
\[\Delta(g\Gamma_N)=\max_{\mathbf v\in g\Gamma_N\mathbf e_1}\log\left(\frac{1}{\|\mathbf v\|}\right).\]
Note that $\Delta^{-1}([T,\infty])=X_N^{\eps}$ for $\eps= e^{-T}$ so that
\[
\Psi_{\Delta}(T)=\mu\left(\Delta^{-1}([T,\infty])\right)=\mu\left(X_N^{e^{-T}}\right).
\]
Let \[
a(t)=\diag{e^{t/m},\dots,e^{t/m}, e^{-t/n},\dots, e^{-t/n}}.
\]

The following 
Borel--Cantelli lemma 
is often attributed to W.\ Schmidt (cf.\ \cite[Lemma 10, Chapter I]{sprindzhuk}):
\begin{theorem}
\label{thm:cantelli}
Let $(Y,\mu)$ be a measure space and $f_k:Y\to\bR_{\geq0}$ integrable, $0\leq g_k\leq h_k\leq 1$ two sequences of numbers, and suppose there exists $c>0$
\begin{equation}\label{eq:assumptioncantelli}
\int \left(\sum_{k\in[M,K]} f_k(x)- \sum_{k\in[M,K]}g_k\right)^2d\mu(x)\leq c\sum_{k\in[M,K]} h_k
\end{equation}
for any $M<K$, then for any $\eps>0$, if $E_K=\sum_{k\in[1,K]}g_k$, $F_K=\sum_{k\in[1,K]}h_k$ for $\mu$-a.e.\ x,
\begin{equation}\label{eq:conclusioncantelli}
\sum_{k\in[1,K]} f_k(x)= E_K+\cO_\eps\left(F_K^{1/2+\eps}\right).
\end{equation}
\end{theorem}

For any $g\in\SLR[d]$ denote $\|g\|=\max{\left(\|g\|_\infty,\|g\|_\infty^{-1}\right)}$ and for any $L^2(X_N,\mu)$ let $g\acts f$ denote the function $x\mapsto f\left(g^{-1}x\right)$. Recall that the action of $\SLR[d]$ on $X_N$ has a spectral gap. This will be applied using the following special case of \cite[Corollary]{chh}. 
\begin{theorem}
\label{thm:mixing}
There exist constants $c,\delta>0$ such that any $g\in\SLR[d]$ and any $\SOR[d]$-invariant $f_1,f_2\in L^2(X_N,\mu)$ satisfy
\begin{equation}\label{eq:effmixing}
|\langle g\acts f_1,f_2\rangle-\mu(f_1)\mu(f_2)|\leq c\|g\|^{-\delta}\|f_1\|_2\|f_2\|_2.
\end{equation}
\end{theorem}

Theorem~\ref{thm:cantelli} and Theorem~\ref{thm:mixing} are used together to prove the following corollary, which is a special case of \cite[Theorem 4.3]{kleinbockmargulis2}. Our argument  follows \cite[Proof of Theorem 1.2]{kleinbockmarguliserratum}.

\begin{corollary}\label{cor:borelcantelli}
If $\{r_t\}_{t\in \N}\subset [1,\infty)$ satisfy
\begin{equation}\label{eq:sumdivergence}
\sum_{t=1}^\infty\Psi_{\Delta}(r_t)=\infty
\end{equation}
then for almost all $x\in X_N$ there exist infinitely many $t\in\N$ such that $\Delta(a(t)x)\geq r_t$.
\end{corollary}

\begin{proof}
We wish to apply Theorem~\ref{thm:cantelli} to $Y=X_N$ with
$$
f_k=a(k)^{-1}\acts \chi_k
$$
and
$$
g_k=h_k=\mu\left(X_N^{\eps_k}\right)=\Psi_{\Delta}(r(k))\leq1,
$$
where $\chi_k$ is the characteristic function of $X_N^{\eps_k}$ and
$$
\eps_k=e^{-r(k)}.
$$
Using the notation of Theorem~\ref{thm:cantelli}, we have
$$
E_K=F_K=\sum_{k=1}^K \mu\left(X_N^{{\eps_k}}\right),
$$
which diverges by \eqref{eq:sumdivergence}. Assume that \eqref{eq:assumptioncantelli} holds. Then \eqref{eq:conclusioncantelli} implies, in particular, that $\sum_{k=1}^K f_k(x)$ diverges for almost every $x$. Therefore, $a(k)x\in X_N^{\eps_k}$, or, equivalently, $\Delta(a(k)x)\geq r_k$ for infinitely many $k$, which is the claim of the corollary.

In order to verify \eqref{eq:assumptioncantelli}, first note that $X_N^\eps$ is an $\SOR[d]$-invariant set for any $\eps$, and $\|\chi_k\|_2=\mu(X_N^{\eps_k})^{\frac12}$ for any $k\in\bbn$. Therefore, by \eqref{eq:effmixing}:
\begin{align*}
\left|\sum_{k,\ell\in[M,K]}\right.&\left.\vphantom{\sum_k}\langle a(\ell)a(k)^{-1}\acts\chi_k, \chi_\ell\rangle  - \mu(X_N^{\eps_k})\mu(X_N^{\eps_\ell})\right| \leq \\
& c \sum_{k,\ell\in[M,K]} \|a(\ell)a(k)^{-1}\|^{-\delta} \mu(X_N^{\eps_k})^{\frac12}\mu(X_N^{\eps_\ell})^{\frac12} = \\
& c\sum_{k,\ell\in[M,K]:\eps_k<\eps_\ell} e^{\frac{-\delta|\ell-k|}{\min{(m,n)}}} \mu(X_N^{\eps_\ell})
+c\sum_{k,\ell\in[M,K]:\eps_k\geq\eps_\ell} e^{\frac{-\delta|\ell-k|}{\min{(m,n)}}} \mu(X_N^{\eps_k}) = \\
& c\sum_{k,\ell\in[M,K]} e^{\frac{-\delta|\ell-k|}{\min{(m,n)}}} \mu(X_N^{\eps_\ell})
+c\sum_{k,\ell\in[M,K]} e^{\frac{-\delta|\ell-k|}{\min{(m,n)}}} \mu(X_N^{\eps_k}) \leq \\
& 2c \sum_{k\in[M,K]} \mu(X_N^{\eps_k}) \sum_{\ell\in[M,K]}e^{\frac{-\delta|\ell-k|}{\min{(m,n)}}} \leq c' \sum_{k\in[M,K]} \mu(X_N^{\eps_k}),
\end{align*}
where $c'=4c\sum_{\ell=0}^\infty e^{\frac{-\delta\ell}{\min{(m,n)}}}$ (cf. \cite[Section 4.4]{kleinbockmargulis2}).
\end{proof}

We need one final tool, which associates to any approximation function $\psi$ a rate function $\cD_{m,n}(\psi)=r$ (named after S.G. Dani) and vice versa \cite[Lemma 8.3]{kleinbockmargulis2}:

\begin{lemma}\label{lem:changeofvar}
Given $x_0>0$ and $\psi:[x_0,\infty)\to(0,\infty)$ be a non-increasing continuous 
function, there exists a unique continuous function $r:[t_0,\infty)\to\R$ with $t_0=\{\frac{m}{m+n}\log x_0-\frac{n}{m+n}\log\psi(x_0)\}$ satisfying
\begin{align}
\lambda(t) : & =  t-nr(t) && \text{ is strictly increasing and unbounded,}  \label{eq:dania}\\
L(t) : & =  t+mr(t)  && \text{ is nondecreasing, and}  \label{eq:danib}\\
\psi(e^{t-nr(t)}) & =  e^{-t-mr(t)} && \text{ for all } t\geq t_0.  \label{eq:danic}
\end{align}
Conversely, given $t_0\in\R$ and a continuous function $r:[t_0,\infty)\to\R$ such that \eqref{eq:dania} holds, there exists a unique continuous non-increasing function $\psi:[x_0,\infty)\to(0,\infty)$ with $x_0=e^{t_0-nr(t_0)}$ satisfying \eqref{eq:danic}.

Furthermore,
\begin{align}\label{eq:danid}
\int_{x_0}^\infty \psi(x) dx <\infty && \text{   if and only if   } && \int_{t_0}^\infty e^{-(m+n)r(t)} dt <\infty.
\end{align}
\end{lemma}

Using the notation of the previous lemma it is possible to state an appropriate extension of Dani's correspondence (cf.\ \cite[Theorem 8.5]{kleinbockmargulis2}):

\begin{lemma}\label{lem:danicorrespondence}
$x\in X_N$ is $\gamma$-congruence-$\psi$-approximable if and only if there exist arbitrarily large positive $t$ such that $\Delta(a(t)x^\gamma)\geq r_t$.
\end{lemma}
\begin{proof} 
We only prove the forward direction, the proof of the converse is similar and will be omitted.

Assume 
that $x$ is $\gamma$-congruence-$\psi$-approximable. Let $(\mathbf v,\mathbf w)\in g\gamma \Gamma_N\mathbf e_1$ be a solution to $\|\mathbf v\|^m\leq \psi(\|\mathbf w\|^n)$. Choose $t$ such that $\|\mathbf w\|^n=e^{\lambda(t)}$, where $\lambda$ is as in \eqref{eq:dania} (this is possible if $\|\mathbf w\|$ is sufficiently large). 
It follows from \eqref{eq:danic} that $\|\mathbf v\|^m \leq \psi\left(e^{\lambda(t)}\right) = e^{-t-mr(t)}$,
or, equivalently, $e^{t/m}\|\mathbf v\|\leq e^{-r(t)}$. By definition of $t$, also $e^{-t/n}\|\mathbf w\|=e^{-r(t)}$. Consequently $\|a(t)(\mathbf v, \mathbf w)\|\leq e^{-r(t)}$, hence, $\Delta(a(t)x^\gamma)\geq r_t$.
\end{proof}

\begin{remark}\leavevmode
\label{rem:quasiincreasing}
\item(i) It will be important to note that in view of \eqref{eq:danid}, it is true that $x$ is $\gamma$-congruence-$\psi$-approximable if and only if for any fixed constant $c$ we have $\Delta(a(t)x^\gamma)\geq r_t-c$.

\item(ii) Using Iwasawa coordinates of $\SLR[d]$, one observes that $\Psi_{\Delta}(T)$ is bounded from below and above by $e^{-(m+n)T}$ up to some multiplicative constants for all sufficiently large $T$. See e.g.\ the proof of \cite[Lemma 3.10]{EMM}.

\end{remark}

The following corollary is deduced from Corollary~\ref{cor:borelcantelli} and Lemma~\ref{lem:danicorrespondence} (cf.\ \cite[Theorem 8.2]{kleinbockmargulis2}):
\begin{corollary}\label{thm:psiapproximable}
Let $\psi:[1,\infty)\to(0,\infty)$ be a non-increasing continuous 
function. If $\int_1^\infty\psi(x)dx = \infty$ then $\mu$ almost every $x\in X_N$ is $\gamma$-congruence-$\psi$-approximable for every $\gamma\in\Gamma$.
\end{corollary}
\begin{proof}
\sloppy
Since $\int_1^\infty\psi(x)dx=\infty$, by \eqref{eq:danid}, $\int_{1}^\infty e^{-(m+n)r(t)} dt=\infty$.
It follows from \eqref{eq:danib} that $\sum_{k\in\bN}e^{-(m+n)r(k)}=\infty$. By Remark~\ref{rem:quasiincreasing}(ii) the sum in \eqref{eq:sumdivergence} diverges. Hence, Corollary~\ref{cor:borelcantelli} implies that for almost all $x\in X_N$ there are arbitrarily large $t$ that satisfy $\Delta(a(t)x)\geq r_t$. Call this conull set $X'$. Fix $F$ to be a fundamental domain of $X$, so that $X=\bigcup_{\gamma\in\Gamma}F\gamma$,
and let $\{\gamma_1,\ldots,\gamma_\ell\}\subseteq\Gamma_N/\Gamma$ be a system of coset representatives. Then
$$
E=\bigcup_{i=1}^\ell F\gamma_i
$$
is a fundamental domain of $X_N$. Hence, there exists a conull set $E'\subset E$ (with respect to the Haar of $G$ restricted to $E$) such that $E\Gamma_N\subset X'$. Therefore, for every $1\leq i\leq \ell$, there exists a conull set $F_i\subset F$ (with respect to Haar of $G$ restricted to $F$) such that $F_i\gamma_i\Gamma_N\subset X'$. Let
$$
F' = \bigcap_{i=1}^\ell F_i.
$$
It follows that
$$
X''=\bigcup_{i=1}^\ell F'\gamma_i\Gamma_N
$$
is conull in $X_N$, and any $x\in X''$ is $\gamma$-congruence-$\psi$-approximable for every $\gamma\in\Gamma$.
\end{proof}

The final step is to pass from Corollary~\ref{thm:psiapproximable} to a similar statement about $u(\theta)\Gamma_N$ for Lebesgue almost every $\theta\in M_{m,n}(\R)$. Let $B$ denote the subgroup of $\SLR[d]$ consisting of all lower triangular block matrices of the form \(\mat{B_1&0\\B_2& B_3}\), and let $U$ denote the group generated by $u(\theta)$ as $\theta$ varies in $M_{m,n}(\R)$. The argument of \cite[Section~8.7]{kleinbockmargulis2} relies on the fact that for any $x\in X$ there are $b\in B$ and $u\in U$ such that $x=bu\Gamma$. However, there's no similar decomposition for elements in $X_N$ so we must slightly modify their argument as follows:

\begin{proof}[Proof of Theorem~\ref{thm:KAlongProgressions}]
By the implicit function theorem, the map $(b,u)\mapsto bu$ is a local diffeomorphism from $B\times U$ to $G$. It follows that there exist open neighborhoods $\Omega_B\subset B$ and $\Omega_U\subset U$ such that $\Omega_G = \Omega_B\Omega_U$ is an open neighborhood of $\Id\in\SLR[d]$. Let
$$
U=\bigcup_{k=1}^\infty\Omega_Uu_k
$$
be a cover of $U$ by countably many translates of $\Omega_U$ such that $\Omega_Gu_k\gamma$ is injective in $X_N$  for all $\gamma\in \Gamma_N/\Gamma$. Since $\Omega_Gu_j\gamma\Gamma_N$ is open, it follows from Corollary~\ref{thm:psiapproximable} that almost every point in $\Omega_Gu_k\Gamma_N$ is $\gamma$-congruence-$\psi$-approximable.
By Fubini's theorem, for almost every $u\in\Omega_u$ there exists a set of full measure $\Omega_B'\subset \Omega_B$ such that $buu_k\Gamma_N$ is $\gamma$-congruence-$\psi$-approximable for every $b\in\Omega_B'$. Since $B$ is the weak stable submanifold of $a(t)$, it satisfies
$$
a(t)\Omega_Ba(t)^{-1}\subset\Omega_B.
$$
It follows from the uniform continuity of $\Delta$ that if $bu(\theta)u_k\gamma\Gamma_N\in \Omega_Gu_k\gamma\Gamma_N$ then
$$
\Delta(a(t)bu(\theta)u_k\gamma\Gamma_N)\geq r_t \Longrightarrow \Delta(a(t)u(\theta)u_k\gamma\Gamma_N)\geq r_t-c,
$$
where the constant $c$ depends on the size of $\Omega_B$. In view of Lemma~\ref{lem:danicorrespondence} and Remark~\ref{rem:quasiincreasing}(i) we conclude that $a(t)u(\theta)u_k\Gamma_N$ is $\gamma$-congruence-$\psi$-approximable for almost all $u\in\Omega_U$. Taking countable intersection over $\gamma$ and $u_k$ shows the theorem.
\end{proof}

\subsection*{Acknowledgements}
The authors wish to thank Barak Weiss for pointing out to them the connection between congruence conditions and cusps of $X_N$.


\vspace*{\fill}
\noindent Erez Nesharim \\ Faculty of Mathematics\\ Technion \\ Israel \\ {\tt ereznesh@gmail.com }\\ \\
\noindent Rene R\"{u}hr \\ Faculty of Mathematics\\ Technion \\ Israel \\ {\tt rener@campus.technion.ac.il }\\ \\
\noindent Ronggang Shi \\  Shanghai Center for Mathematical Sciences \\ Jiangwan Campus \\ Fudan University \\ No.2005 Songhu Road, Shanghai, 200433, China \\  {\tt ronggang@fudan.edu.cn}\\ \\

\bibliographystyle{siam}
\bibliography{bibdeskfile}

\begin{thebibliography}{10}

\bibitem{adiceam2015progressions}
{\sc F.~Adiceam}, {\em Rational approximation and arithmetic progressions},
  Int. J. Number Theory, 11 (2015), pp.~451--486.

\bibitem{beresnevichVelaniZeroOneLaws}
{\sc V.~Beresnevich and S.~Velani}, {\em A note on zero-one laws in metrical
  {D}iophantine approximation}, Acta Arith., 133 (2008), pp.~363--374.

\bibitem{beresnevichVelani2010}
\leavevmode\vrule height 2pt depth -1.6pt width 23pt, {\em Classical metric
  {D}iophantine approximation revisited: the {K}hintchine-{G}roshev theorem},
  Int. Math. Res. Not. IMRN,  (2010), pp.~69--86.

\bibitem{chh}
{\sc M.~Cowling, U.~Haagerup, and R.~Howe}, {\em Almost {$L^2$} matrix
  coefficients}, J. Reine Angew. Math., 387 (1988), pp.~97--110.

\bibitem{dani}
{\sc S.~G. Dani}, {\em Divergent trajectories of flows on homogeneous spaces
  and diophantine approximation}, J. reine angew. Math, 359 (1985), p.~102.

\bibitem{EMM}
{\sc A.~Eskin, G.~Margulis, and S.~Mozes}, {\em Upper bounds and asymptotics in
  a quantitative version of the {O}ppenheim conjecture}, Ann. of Math. (2), 147
  (1998), pp.~93--141.

\bibitem{harman}
{\sc G.~Harman}, {\em Metric {D}iophantine approximation with two restricted
  variables. {I}. {T}wo square-free integers, or integers in arithmetic
  progressions}, Math. Proc. Cambridge Philos. Soc., 103 (1988), pp.~197--206.

\bibitem{hartmanszusz}
{\sc S.~Hartman and P.~Sz\"usz}, {\em On congruence classes of denominators of
  convergents}, Acta Arith., 6 (1960), pp.~179--184.

\bibitem{kleinbockmarguliserratum}
{\sc D.~Kleinbock and G.~Margulis}, {\em Erratum to: Logarithm laws for flows
  on homogeneous spaces}, Inventiones mathematicae,  (2017), pp.~1--8.

\bibitem{kleinbockmargulis2}
{\sc D.~Kleinbock and G.~A. Margulis}, {\em Logarithm laws for flows on
  homogeneous spaces}, Inventiones mathematicae, 138 (1999), pp.~451--494.

\bibitem{kleinbock2017pointwise}
{\sc D.~Kleinbock, R.~Shi, and B.~Weiss}, {\em Pointwise equidistribution with
  an error rate and with respect to unbounded functions}, Mathematische
  Annalen, 367 (2017), pp.~857--879.

\bibitem{kleinbock2008dirichlet}
{\sc D.~Kleinbock and B.~Weiss}, {\em Dirichlet's theorem on diophantine
  approximation and homogeneous flows}, Journal of Modern Dynamics, 2 (2008),
  pp.~43--62.

\bibitem{kleinbock1998flows}
{\sc D.~Y. Kleinbock}, {\em Flows on homogeneous spaces and diophantine
  properties of matrices}, Duke mathematical journal, 95 (1998), pp.~107--124.

\bibitem{shi2015expanding}
{\sc R.~Shi}, {\em Expanding cone and applications to homogeneous dynamics,
  with an appendix by {R}.~{R}\"uhr and the author}, arXiv preprint
  arXiv:1510.05256,  (2015).

\bibitem{sprindzhuk}
{\sc V.~G. Sprindzhuk}, {\em Metric theory of Diophantine approximations},
  Halsted Press, 1979.

\end{thebibliography}
\end{document}